\documentclass[reqno,a4paper]{amsart}
\usepackage[english,activeacute]{babel}
\usepackage{amssymb,amsmath,amsthm,amsfonts,mathrsfs,scalefnt,graphicx,graphics,color,hyperref,}
\usepackage{lmodern}
\usepackage{tikz}
\usetikzlibrary{trees}
\usetikzlibrary{shapes}

\theoremstyle{plain}
\newtheorem{theorem}{Theorem}[section]
\newtheorem{lemma}{Lemma}[section]

\newtheorem{definition}{Definition}[section]

\newtheorem{remark}{Remark}[section]
\numberwithin{equation}{section}

\newcommand{\gn}{{\lfloor\gamma N\rfloor}}

\newcommand{\Nb}  {{\mathbb N}}
\newcommand{\Rb}  {{\mathbb R}}

\newcommand{\As} {{\mathcal A}}

\newcommand{\Fs} {{\mathcal F}}

\newcommand{\Ps} {{\mathcal P}}

\newcommand{\Ys} {{\mathcal Y}}

\newcommand{\al}{\alpha}

\renewcommand{\phi}{\varphi}

\newcommand{\la}{\lambda}

\newcommand{\ind}{1\!\kern-1pt \mathrm{I}}
\newcommand{\rsto}{]\!\kern-1.8pt ]}
\newcommand{\lsto}{[\!\kern-1.7pt [}

\newcommand\F{\mbox{I\kern-2pt F}}
\newcommand\1{{\bf 1}}

\title[Fractional binary markets under transaction costs]{Critical transaction costs and 1-step asymptotic arbitrage in fractional binary markets}
\author{Fernando Cordero}
\address{Faculty of Technology, University of Bielefeld, Universit\"{a}tsstr. 25, 33615 Bielefeld, Germany}
\email{fcordero@techfak.uni-bielefeld.de}

\author{Lavinia Perez-Ostafe}
\address{Department of Mathematics, ETH Zurich, R\"{a}mistrasse 101, 8092 Zurich, Switzerland}
\email{lavinia.perez@math.ethz.ch}
\date{\today}%

\begin{document}


\begin{abstract}
We study the arbitrage opportunities in the presence of transaction costs in a sequence of binary markets approximating the fractional Black-Scholes model. This approximating sequence was constructed by Sottinen and named fractional binary markets. Since, in the frictionless case, these markets admit arbitrage, we aim to determine the size of the transaction costs needed to eliminate the arbitrage from these models. To gain more insight, we first consider only 1-step trading strategies and we prove that arbitrage opportunities appear when the transaction costs are of order $o(1/\sqrt{N})$. Next, we characterize the asymptotic behavior of the smallest transaction costs $\la_c^{(N)}$, called ``critical'' transaction costs, starting from which the arbitrage disappears. Since the fractional Black-Scholes model is arbitrage-free under arbitrarily small transaction costs, one could expect that $\la_c^{(N)}$ converges to zero. However, the true behavior of $\la_c^{(N)}$ is opposed to this intuition. More precisely, we show, with the help of a new family of trading strategies, that $\la_c^{(N)}$ converges to one. We explain this apparent contradiction and conclude that it is appropriate to see the fractional binary markets as a large financial market and to study its asymptotic arbitrage opportunities. Finally, we construct a $1$-step asymptotic arbitrage in this large market when the transaction costs are of order $o(1/N^H)$, whereas for constant transaction costs,  we prove that no such opportunity exists.
\end{abstract}
\maketitle
\section{Introduction}\label{1}
Significant academic research shows that the use of price models driven by fractional Brownian motion substantially increased, even if this was until recently rejected as these models are not free of arbitrage (see \cite{Rog}). Despite this drawback, these models are thought to describe real world markets in a better way. This is because, when the Hurst parameter is strictly bigger than $1/2$, the fractional Brownian motion exhibits self-similarity and long-range dependence, properties that were observed in empirical studies of financial time series (see \cite{Con} and \cite{WiTaTe} for a discussion on the relevance of these properties in financial modelling). Moreover, when one introduces transaction costs, the arbitrage opportunities disappear (see for ex. \cite{G:R:S:2008}), which makes it then possible to deduce a valuation theory that is built on no arbitrage arguments.

A typical example of such market is the fractional Black-Scholes model, which is in fact a Black-Scholes type model where the randomness of the risky asset comes from a fractional Brownian motion. There is extensive literature around the properties of this model, in particular, explicit arbitrage opportunities can be found in \cite{Nu}, \cite{SoVa}, \cite{Che} and \cite{BeSoVa}. Additionally, as shown by Sottinen in \cite{Sotti}, the fractional Black-Scholes model can be approximated by a sequence of binary models, called ``fractional binary markets''. This result is based on an analogue of the Donsker's theorem, which, in this case, states that the fractional Brownian motion can be approximated by a ``disturbed'' random walk. The markets in this approximating sequence will be our object of study in this paper. The motivation lies not only in the fact that these models behave asymptotically as the fractional Black-Scholes model, but also in their simplicity, coming from their binary structure.

A $N$-step binary market is a market in which the stock price $(S_n)_{n=0}^N$ is an adapted stochastic process with strictly positive values and such that at time $n$ the stock price evolves from $S_{n-1}$ to either $\alpha_n\, S_{n-1}$, in which case we say that the stock price process takes a step up, or $\beta_n\,   S_{n-1} $, in which case we say that the stock price process takes a step down, where $\beta_n <\alpha_n$. In addition, the parameters $\alpha_n$ and $\beta_n$, for $n\in\{1,..,N\}$, depend only on the past, which means that they can be seen as real valued functions on $\{-1,1\}^{n-1}$. One advantage of working with binary models is that the study of arbitrage, in the frictionless case, reduces to the study of a family of conditions imposed on the nodes of the binary tree $\cup_{n=1}^{N}\{-1,1\}^{n-1}$ (see \cite{Dzh}). When one of these conditions is not verified in a node of the binary tree, we call this node an arbitrage point. In this context, we call $N$-fractional binary market the $N$-step binary market in the sequence of fractional binary markets. Sottinen showed in \cite{Sotti} that the fractional binary markets without friction admit arbitrage and such an opportunity is explicitly constructed using the path information starting from time zero. Moreover, in the recent work \cite{CKP}, it was proved that the asymptotic proportion of arbitrage points in the fractional binary markets is strictly positive and a characterization of that quantity, in terms of the Hurst parameter, is provided. 

In the present paper we aim to analyse the sensibility of the arbitrage condition to the presence of proportional transaction costs, first referring to each fixed $N$-fractional binary market and then considering the whole sequence as a large financial market. In the latter case, the notion of arbitrage is replaced by the concept of ``asymptotic arbitrage'', which was introduced by Kabanov and Kramkov in \cite{Kab:Kra:94} and \cite{Kab:Kra:1998} and further studied in \cite{K:S:1996} and \cite{K:Sch:1996} for frictionless markets and in \cite{Le:Os} and \cite{Kl:Le:Pe} for the case of transaction costs. 

We first treat independently each market in the sequence and we look for arbitrage opportunities under transaction costs. More precisely, we study the smallest transaction costs $\la_c^{(N)}$, called ``critical'' transaction cost, starting from which the $N$-fractional binary market is arbitrage free. Since the arbitrage opportunities disappear when arbitrarily small transaction costs are introduced in the fractional Black-Scholes model, one can expect that, asymptotically, the same behavior occurs also in the fractional binary markets. This would be the case if the sequence of critical transaction costs converges to zero. Surprisingly, the behavior will be opposite to what is expected as we will show that, in fact, $\la_c^{(N)}$ converges to one and not to zero. We will approach this problem in two steps. First, in order to get some intuition, we study the existence of arbitrage under transaction costs in our models, when we restrict ourselves to use only $1$-step self-financing strategies. We know from \cite{Sotti} that, if the price process takes steps down (up) till some time $n_0$, with $n_0$ big enough, then the price will decrease (respectively, increase) from time $n_0$ to time $n_0+1$. An arbitrage is then explicitly constructed, in the frictionless case, by going short (respectively, long) at time $n_0$ and liquidating at time $n_0+1$. Using the same idea, we construct 1-step arbitrage opportunities in the fractional binary markets, which are subject to transaction costs of order $o(1/\sqrt{N})$. In a second step, we look to more general, but still elementary, self-financing strategies. The key point in the construction of the new arbitrage opportunities is that, if the stock price process takes only steps down a proportional number of $N$ times, $\gamma N$, then starting from this time the stock price will decrease a proportional number of $N$ times, $P_\gamma N$. We can then construct an arbitrage (in the frictionless case) by short-selling one unit of stock at time $\gamma N$ and then liquidating the position at the later time $\gamma N +P_\gamma N$. Next, we extend the previous construction, in the natural way, to the case with friction, and we prove that, if the transaction costs $\lambda_N$ are smaller than $1-e^{-C\sqrt{N}}$, for some constant $C>0$, then the corresponding self-financing strategy provides an arbitrage opportunity in the $N$-fractional binary market. As a consequence, we deduce that $\la_c^{(N)}$ converges to $1$. This does not contradict the absence of arbitrage in the fractional Black-Scholes with friction and we provide an explanation for that. In this context, the notion of asymptotic arbitrage seems to reflect better the abovementioned property of the limit market.
Following this line, we construct a 1-step asymptotic arbitrage of first kind in the fractional binary markets when the transaction costs are of order $o(1/N^H)$ and we show that if we consider constant transaction costs these possibilities disappear.

The paper is organized as follows. In Section~\ref{2}, we start by recalling some notations and definitions concerning binary markets that we will use along this work. When talking about arbitrage opportunities with proportional transaction costs, the concepts of arbitrage, self-financing system and critical transaction cost are introduced. Moreover, the concept of asymptotic arbitrage is recalled. We end this part with a brief presentation of fractional binary markets and recall some estimates obtained in \cite{CKP} for the involved quantities. In Section~\ref{3}, Section~\ref{4} and Section \ref{5} are concentrated the main results. In Section~\ref{3}, we construct a sequence of 1-step self-financing strategies leading to arbitrage opportunities in the fractional binary markets, when they are subject to transaction costs converging fast enough to zero. In Section~\ref{4}, we show that the sequence of critical transaction costs associated to the fractional binary markets converges to $1$, by constructing a new well-chosen sequence of trading strategies leading to arbitrage opportunities under ``big'' transaction costs. We end this paper with Section~\ref{5}, which contains some conclusions on the existence of 1-step asymptotic arbitrage in the sequence of fractional binary markets under small transaction costs.


\section{Preliminaries}\label{2}

\subsection{Binary markets}\label{bmao}
Let $(\Omega,\Fs,{(\Fs_n)}_{n=0}^N, P)$ be a finite filtered probability space. By a binary market we mean a market in which two assets (a bond $B$ and a stock $S$) are traded at successive times $t_0=0<t_1<\cdots<t_N$. The evolution of the bond and stock is described, for $n\in\{1,...,N\}$, by:
\begin{equation}\label{stock}
 B_n=(1+r_n)B_{n-1}\,\textrm{ and }\,S_n= \left(a_n+(1+X_n)\right)S_{n-1},
\end{equation}
where $r_n$ and $a_n$ are the interest rate and the drift of the stock. Here, $B_n$ and $S_n$ denote the value of $B$ and $S$ in the time interval $[t_n,t_{n+1})$. The value of $S$ at time $0$ is given by a constant, i.e. $S_0=s_0$. From now on we assume, for the sake of simplicity, that the bond plays the role of a num\'eraire, and, in this case, that it is equal to $1$ at every time $n$ ($r_n=0$). The process $(X_n)_{n=0}^N$ is an adapted stochastic process starting at $X_0=x_0$ and such that, at each time $n$, $X_n$ can take only two possible values $u_n$ and $d_n$ with $d_n<u_n$. When $X_n$ equals $u_n$, we say that the price process takes a step up, and when $X_n$ is given by $d_n$, we say that the price process takes a step down. While $a_n$ from \eqref{stock} is deterministic, the values of $u_n$ and $d_n$ may depend on the path of $X$ up to time $n-1$. Note that the history of the process before time $n$ can be encoded by a vector in ${\{-1,1\}}^{n-1}$, where having $1$ (respectively $-1$) in position $k$ means that at time $k$ the process takes a step up (respectively down). Consequently, the parameters $u_n$ and $d_n$ can be seen as real valued functions on ${\{-1,1\}}^{n-1}$ ($u_1$ and $d_1$ are constants). 

\subsection{Arbitrage opportunities under transaction costs}\label{2.2}
From now on we consider binary markets $S$ as introduced in Section~\ref{bmao} but with proportional transaction costs $\lambda\in[0,1]$, meaning that the bid and ask price of the stock $S$ are modeled by the processes ${((1-\lambda)S_n)}_{n=0}^N$ and ${(S_n)}_{n=0}^N$ respectively. More precisely, we consider without loss of generality that we pay $\la$ transaction costs only when we sell, and not when we buy.

The notion of arbitrage can be seen in an intuitive way as the possibility to make a profit in a financial market without risk and without net investment of capital. This idea is formalized for our framework, i.e for a binary market, where the bond plays the role of a num\'{e}raire, in the following definitions.

\begin{definition}[$\lambda$-self-financing strategy]\label{sfs}
 Given $\lambda\in[0,1]$, a $\lambda$-self-financing strategy for the process ${(S_n)}_{n=0}^N$ is an adapted process $\phi={(\phi_n^0,\phi_n^1)}_{n=-1}^N$ satisfying, for all $n\in\{0,...,N\}$, the following condition:
  \begin{equation}\label{selffin}
  \phi_n^0-\phi_{n-1}^0\leq -{(\phi_n^1-\phi_{n-1}^1 )}^+\,S_n\, +\, (1-\lambda)\,{(\phi_n^1-\phi_{n-1}^1 )}^-\,S_n.
 \end{equation}
Here $\phi^0$ denotes the number of units we hold in the bond and $\phi^1$ denotes the number of units in the stock.
\end{definition}

\begin{definition}[$\lambda$-arbitrage]\label{arb} Given $\lambda\in[0,1]$, we say that the process ${(S_n)}_{n=0}^N$ admits for a $\lambda$-arbitrage, or arbitrage under transaction costs $\lambda$, if there is a $\lambda$-self-financing strategy $\phi={(\phi_n^0,\phi_n^1)}_{n=-1}^N$ starting at $(\phi_{-1}^0,\phi_{-1}^1)=(0,0)$
verifying the following conditions:
\begin{itemize}
 \item $V_N^\la(\phi)\geq 0\quad P-a.s.$
 \item $P\left(V_N^\la(\phi)>0\right)>0$,
\end{itemize}
where $V_n^\la(\phi)$ represents the liquidated value of the portfolio at time $n$ and is given, for each $n\in\{0,\ldots,N\}$, by
$$V_n^\la(\phi):=\phi_n^0+(1-\la)(\phi_n^1)^+S_n-(\phi_n^1)^-S_n.$$ 
If $\la=0$, we simply write $V_n(\phi)$ instead of $V_n^0(\phi)$.
\end{definition}
\begin{remark}\label{1step}
 Along this work, when constructing arbitrage opportunities, special interest will be attributed to a certain kind of self-financing trading strategy, for which one does nothing till a fixed time point, then, depending on the position in the binary tree, we go long or short in the stock and immediately liquidate this position at the very next time step. We call this type of strategy ``1-step'' self-financing strategy.
\end{remark}

In the context of frictionless binary markets, the arbitrage condition can be expressed in terms of the parameters of the model. More precisely, we know by Proposition 3.6.2 in \cite{Dzh} that a binary market excludes arbitrage opportunities if and only if for all $n\in\{1,...,N\}$ and $x\in{\{-1,1\}}^{n-1}$, we have:
\begin{equation}\label{nac1}
d_n(x)<-a_n< u_n(x).
\end{equation}
If for some $n\in\{1,...,N\}$ and $x\in{\{-1,1\}}^{n-1}$, the above condition is not verified, then $x$ is called an ``arbitrage point''. Note that if $x$ is not an arbitrage point, then starting from $x$, the price process takes a step up (down) if and only if its value strictly increases (respectively decreases) at the next time. However, this is not anymore true if $x$ is an arbitrage point.

In the friction case, the arbitrage condition was studied in \cite{CKO}, where the authors provide a characterization of the smallest transaction costs, called ``critical'' transaction costs and denoted by $\la_c$, needed to remove arbitrage opportunities, i.e
\begin{equation}\label{lac}
 \lambda_c=\inf\{\lambda\in[0,1]:\textrm{ there is no $\la$-arbitrage}\}.
\end{equation}
More precisely, from \cite[Corollary 5.1]{CKO} we know that
$$\lambda_c=1-\sup\limits_{Q\in \Ps_1(\Omega)}\rho(Q),$$
where $\Ps_1(\Omega)$ denotes the space of all probability measures on $(\Omega,\Fs)$ and the function $\rho:\Ps_1(\Omega)\rightarrow [0,1]$ is defined by means of the parameters of the model $\{a_n,u_n,d_n\}_{n\geq 1}$. Therefore, the problem of computing $\lambda_c$ leads to solve an optimization problem, which, depending on the structure of the market, can be very difficult. Nevertheless, for 1-step binary markets, $\lambda_c$ can be explicitly computed. In addition, if we decompose a multi-step binary market in 1-step sub-markets, one can obtain the following lower bound for $\lambda_c$ (see \cite[Proposition 3.1]{CKO}):
\begin{align}\label{lowbd}
\lambda_c&\geq1-\min\limits_{n\in\{1,...,N\}}\left\{\min\limits_{x\in{\{-1,1\}}^{n-1}}\left\{(1+a_n+u_n(x))\wedge\frac{1}{1+a_n+d_n(x)}\wedge 1\right\}\right\}.
\end{align}
In general, the study of $\lambda$-arbitrage opportunities is much more complex than only looking for $\la$-arbitrage opportunities in the 1-step sub-binary markets and the previous lower bound can be not very accurate. We also point out, that all the aforementioned results are obtained by means of a dual approach, in which the existence of arbitrage opportunities (resp. $\la$-arbitrage opportunities) is related to the existence of an equivalent martingale measure (resp. a $ \la$-consistent price system). In particular, this dual approach does not provide explicit arbitrage opportunities. 

Our approach will be based on the construction of explicit families of self-financing strategies. In this respect, an important role will be played by the following notion.

\begin{definition}[Self-financing system]\label{sfsm}
A self-financing system for the process ${(S_n)}_{n=0}^N$ is a family $\Phi=\{\phi(\la)\}_{\la\in[0,1]}$, where for all $\la\in[0,1]$, $\phi(\la)$ is a $\la$-self-financing strategy. Given a self-financing system $\Phi=\{\phi(\la)\}_{\la\in[0,1]}$, we define
\begin{equation}\label{laPhi}
 \la(\Phi):=\inf\{\la\in[0,1]:\phi(\la)\textrm{ is not a $\la$-arbitrage}\}.
\end{equation}
\end{definition}
\begin{remark}\label{ilplc}
Note that if $\Phi=\{\phi(\la)\}_{\la\in[0,1]}$ is a self-financing system, then $\la(\Phi)$ can be expressed as follows
$$\la(\Phi)=\inf\{\la\in[0,1]:P(V_N^\la(\phi(\la))>0)=0\textrm{ or }P(V_N^\la(\phi(\la))<0)>0\}.$$
In addition, from their definitions, we have that $\la(\Phi)\leq \la_c$. Therefore, the construction of self-financing systems provides lower bounds for $\la_c$. In particular, if a self-financing system $\Phi=\{\phi(\la)\}_{\la\in[0,1]}$ verifies that $\la(\Phi)>0$, we can conclude that, for all $\la\in[0,\la(\Phi))$, $\phi(\la)$ provides a $\la$-arbitrage opportunity. 
\end{remark}

\subsection{Asymptotic arbitrage under transaction costs}
In this paper, we don't limit our study only to the arbitrage opportunities for an $N$-fractional binary market, but we are also interested in obtaining answers to this problem when the time grid of the approximating sequence of fractional binary markets becomes finer and finer, i.e. $N\to\infty$. We first look to this problem by studying the limit behavior of the sequence of critical transaction costs associated to the fractional binary markets. In a second approach, we interpret the sequence of fractional binary markets as a large financial market and, in this case, replace the notion of arbitrage, as presented in Section~\ref{2.2}, by a new concept. Kabanov and Kramkov defined it as ``asymptotic arbitrage'', \cite{Kab:Kra:94}, and distinguished between two kinds: asymptotic arbitrage of the first kind (AA1) and asymptotic arbitrage of the second kind (AA2). We recall now their definitions. For a detailed presentation we refer the reader to \cite{K:S:1996}, \cite{K:Sch:1996} and \cite{Kab:Kra:1998} for the frictionless case and to \cite{Kl:Le:Pe} for markets with friction. Consider a sequence of markets ${\{S^N\}}_{N\geq 1}$, where $S^N=(S_n^N)_{n=0}^N$, and fix a sequence ${\{\lambda_N\}}_{N\geq 1}$ of real numbers $0<\lambda_N<1$.

\begin{definition}\label{AA1d}
 There exists an asymptotic arbitrage of the first kind (AA1)  with transaction costs $\la_N$  if there exists a subsequence of markets (again denoted by $N$) and self-financing strategies $\phi^N=(\phi^{N,0}, \phi^{N,1})$ with zero endowment for $S^N$ such that
\begin{enumerate}
\item{($c_N$-admissibility condition)} For all $i=0,\ldots,N$, $$V^{\la_N}_i(\phi^N)\geq-c_N,$$
\item $\lim_{N\to\infty}P^N(V_{N}^{\la_N}(\phi^N)\geq C_N)>0$
\end{enumerate}
where $c_N$ and $C_N$ are sequences of strictly positive real numbers with $c_N\to0$ and $C_N\to\infty$. 
\end{definition}

\begin{definition}\label{AA2d}
 There exists an asymptotic arbitrage of the second kind (AA2) with transaction costs $\la_N$ if there exists a subsequence of markets (again denoted by $N$) and self-financing strategies $\phi^N=(\phi^{N,0}, \phi^{N,1})$ with zero endowment for $S^N$ and $\alpha>0$ such that
\begin{enumerate}
\item{($1$-admissibility condition)} For all $i=0,\ldots,N$, $$V^{\la_N}_i(\phi^N)\geq-1,$$
\item $\lim_{N\to\infty}P^N(V^{\la_N}_{N}(\phi^N)\geq \alpha)=1$.
\end{enumerate}
\end{definition}

These two types of asymptotic arbitrage can be intuitively explained as follows. AA1 can be seen as the opportunity of getting arbitrarily rich with strictly positive probability by taking an arbitrarily small risk. AA2 gives the opportunity of gaining at least something, even if only a very small amount, with probability arbitrarily close to one, while taking the risk of losing an amount of money which is uniformly bounded in time. The key difference between the two notions is that in the latter, although the chance of profit is very likely, the risk is not vanishing any more.

\subsection{Fractional binary markets}
Sottinen introduces in \cite{Sotti} the fractional binary markets as a sequence of binary markets approximating the fractional Black-Scholes model. By the latter we mean a Black-Scholes type model in which the randomness of the risky asset comes from a fractional Brownian motion, i.e. the dynamics of the bond and stock are given by:
\begin{equation}\label{fbse}
 dB_t=r(t)\,B_t\, dt\quad\textrm{and}\quad dS_t^H=(a(t)+\sigma\,dZ^H_t)\, S_t^H,
\end{equation}
where $\sigma>0$ is a constant representing the volatility and $Z^H$ is a fractional  Brownian motion of Hurst parameter $H>1/2$. The functions $r$ and $a$ are deterministic and represent the interest rate and the drift of the stock. We assume in the sequel that $r=0$ and that $a$ is continuously differentiable. Since all the parameters of the model are understood to depend on the Hurst parameter $H$, we will avoid to mention this dependence. 

For the sake of simplicity, we make use of the results obtained in \cite{CKP} to provide here an alternative, but equivalent definition of the fractional binary market. First, we consider a sequence of i.i.d. random variables $\{\xi_i\}_{i\geq 1}$ such that $$P(\xi_1=-1)=P(\xi_1=1)=1/2,$$ 
and we define the filtration $\{\Fs_{i}:=\sigma(\xi_1,\dots,\xi_i)\}_{i\geq 1}$.

For each $N>1$, the $N$-fractional binary market is the binary market in which the bond and the stock are traded at times $\{0, \frac{1}{N},...,\frac{N-1}{N}, 1\}$ under the dynamics:
$$B_n^{(N)}=1\quad\textrm{and}\quad S_n^{(N)}=\left(1+a_n^{(N)}+\frac{X_n}{N^H}\right)\, S_{n-1}^{(N)},\quad 1\leq n\leq N$$
where $a_n^{(N)}=\frac{1}{N}a(n/N)$ and $S_0^{(N)}=s_0$. As shown in \cite{CKP}, the process $(X_n)_{n\geq 1}$ can be expressed as
\begin{equation}\label{scale}
X_n:=\sum\limits_{i=1}^{n-1}j_n(i)\,\xi_i+g_n\xi_n,
\end{equation}
where
$$j_n(i):=\sigma\, c_H \left(H-\frac12\right)\int\limits_{i-1}^{i}x^{\frac{1}{2}-H}\left(\int\limits_0^1 (v+n-1)^{H-\frac{1}{2}} (v+n-1-x)^{H-\frac{3}{2}}dv\right) dx,$$
and
$$g_n:=\sigma\, c_H \left(H-\frac12\right)\int\limits_{n-1}^{n}x^{\frac{1}{2}-H}(n-x)^{H-\frac{1}{2}}\left(\int\limits_0^1 (y(n-x)+x)^{H-\frac{1}{2}}y^{H-\frac{3}{2}}dy\right)dx,$$
with $c_H$ a normalizing constant given by
\begin{equation}\label{ch}
c_H:=\sqrt{\frac{2H\,\Gamma\left(\frac{3}{2}-H\right)}{\Gamma\left(H+\frac{1}{2}\right)\,\Gamma(2-2H)}}.
\end{equation}

From \eqref{scale}, we see that $X_n$ is the sum of a process depending only on the information until time $n-1$ and a process depending only on the present. More precisely, $X_n=\Ys_n+g_n\xi_n$, where 
$$\Ys_n:=\sum_{i=1}^{n-1}j_n(i)\xi_i.$$
In the same way, we define for each $(x_1,\dots,x_{n-1})\in{\{-1,1\}}^{n-1}$
$$\Ys_n(x_1,\dots,x_{n-1}):=\sum_{i=1}^{n-1}j_n(i)x_i.$$
Note that from definition $\Ys_n=\Ys_n(\xi_1,\ldots,\xi_{n-1})$.

Using these notations, the parameters of the binary market $u_n^{(N)}$ and $d_n^{(N)}$, already introduced in Section~\ref{bmao}, can be expressed as functions on ${\{-1,1\}}^{n-1}$ by setting, for $\vec{x}\in{\{-1,1\}}^{n-1}$:
$$u_n^{(N)}(\vec{x}):=\frac{\Ys_{n}(\vec{x})+g_n}{N^H}\quad\textrm{ and }\quad d_n^{(N)}(\vec{x}):=\frac{\Ys_{n}(\vec{x})-g_n}{N^H}.$$ 

In order to simplify the presentation, we sometimes use the notation $\vec{\xi}_k$ to denote the random vector $(\xi_1,\ldots,\xi_k)$ in $\{-1,1\}^k$. We also use $\vec{1}_k$ to denote the vector in $\Rb^k$ with all its coordinates equal to $1$.

\subsubsection{Some useful estimations}\label{est}
We briefly recall some estimations obtained in \cite{CKP} for the quantities involved in the definition of the fractional binary markets, i.e., $a_n^{(N)}$, $j_n$ and $g_n$. We avoid the proofs and we invite the reader to directly consult \cite{CKP}.

\begin{lemma}\label{ej1}
For all $1\leq i\leq n-1< N$, we have
$$c_* \,(n-1)^{H-\frac{1}{2}}\,I_n(i)\leq j_n(i)\leq c_*\, n^{H-\frac{1}{2}}\,I_n(i),$$
where $c_*:=\sigma c_H$,
$$I_n(i):=\int\limits_{i-1}^i x^{\frac{1}{2}-H}\phi_n(x)dx\quad \textrm{and}\quad \phi_n(x):=(n-x)^{H-\frac{1}{2}} -(n-1-x)^{H-\frac{1}{2}}.$$ 
 \end{lemma}
 

\begin{lemma}\label{eg1}
 For all $1< n\leq N$, we have
 $$g\,\leq g_n\leq  g\,\left(1+\frac{1}{n-1}\right)^{H-\frac{1}{2}}\leq g\,2^{H-\frac{1}{2}},$$
 where $g:=\frac{\sigma c_H}{H+\frac12}$.
 In particular, $\lim_{n\rightarrow \infty}g_n=g.$
 \end{lemma}
 For the drift term $a_n^{(N)}$, it is straightforward from its definition and the continuity of the function $a$ that:
\begin{equation}\label{ea}
 |a_n^{(N)}|\leq \frac{{||a||}_\infty}{N},\qquad n\in\{1,...,N\}.
\end{equation}
This inequality together with Lemma \ref{eg1} indicates that, given the past $\vec{\xi}_{n-1}$, the contribution of $a_n^{(N)}$ is asymptotically neglictable with respect to the contribution of the last jump $\frac{g_n}{N^H}\,\xi_n$. In addition, since we are interested in asymptotic properties of the fractional binary markets, the problem can be simplified by studying the case without the drift. Therefore, we assume henceforth that $a_n^{(N)}=0$ for all $1\leq n\leq N$.


\section{1-step arbitrage opportunities under small transaction costs}\label{3}

In this section, we assume that each $N$-fractional binary market is subject to proportional transaction costs $\la_N$. We aim to show, using $1$-step $\la_N$-self-financing strategies, the existence of $\lambda_N$-arbitrage opportunities when $\la_N$ converges to $0$ fast enough. We do this in two steps. First we construct, on each $N$-fractional binary market, a $1$-step self-financing system $\Phi^N=\{\phi^N(\la)\}_{\la\in[0,1]}$. Next, we define $\la(\Phi^N)$ as in \eqref{laPhi}.
From its definition, we have that for any $\lambda\in[0,\lambda(\Phi^{N}))$, the self-financing strategy $\phi^{N}(\lambda)$ leads to a $\lambda$-arbitrage in the $N$-fractional binary market. Our problem reduces therefore to studying the asymptotic behavior of the quantity $\lambda(\Phi^{N})$.

To achieve our goal, we follow the same idea as Sottinen in \cite{Sotti}, who constructed 1-step arbitrage opportunities in the frictionless fractional binary markets. More precisely, Sottinen proves in \cite[Theorem 5]{Sotti} that there is $n_H\geq 1$ such that, for all $N\geq n_H$ and for any $n\in\{n_H,\dots,N\}$, 
\begin{equation}\label{albad}
u_{n}^{(N)}(-\vec{1}_{n-1})=\frac1{N^H}\left(-\sum_{i=1}^{n-1}j_n(i)+g_{n}\right)< 0.
\end{equation}
This means that, if the stock price takes jumps only down till time $n-1$, the price process will decrease from time $n-1$ to time $n$. Based on this result, the Sottinen's arbitrage opportunity is constructed in the following way. We chose a level $n_0\geq n_H$ and we don't do anything until time $n_0-2$. If the stock price took only steps down until time $n_0-1$, at that time we short-sell one unit of stock and at the next time we buy one unit of stock. Otherwise we don't do anything. In any case, starting with time $n_0+1$ we don't do anything. From \eqref{albad} it is straightforward to see that this self-financing strategy provides an arbitrage in the frictionless case.

Now, we introduce transaction costs $\lambda$ in the $N$-fractional binary market, we fix $n_0\in\{n_H,\dots,N\}$  and we construct our candidate for a $\la$-arbitrage opportunity $\phi^{N}(\la,n_0)=(\phi^{N,0}(\la,n_0),\phi^{N,1}(\la,n_0))$ as follows:
\begin{itemize}
\item For any time $1\leq i\leq n_0-2$ 
\begin{enumerate}
 \item[] $\phi_{i}^{N,0}(\la,n_0):=\phi_{i}^{N,1}(\la,n_0):=0.$
\end{enumerate}
\item At time $n_0-1$ we short-sell one unit of stock, in which case
\begin{enumerate}
\item[] $\phi_{n_0-1}^{N,0}(\la,n_0):=(1-\la)S^{(N)}_{n_0-1}\1_{\{\vec{\xi}_{n_0-1}=-\vec{1}_{n_0-1}\}},$
\item[] $\phi_{n_0-1}^{N,1}(\la,n_0):=-\1_{\{\vec{\xi}_{n_0-1}=-\vec{1}_{n_0-1}\}}.$
\end{enumerate}
\item At time $n_0$ we liquidate the position, which means buying one unit of stock. In this case 
\begin{enumerate}
\item[]$\phi_{n_0}^{N,0}(\la,n_0):=(1-\la)S^{(N)}_{n_0-1}\1_{\{\vec{\xi}_{n_0-1}=-\vec{1}_{n_0-1}\}}-S^{(N)}_{n_0}\1_{\{\vec{\xi}_{n_0-1}=-\vec{1}_{n_0-1}\}},$
\item[]$\phi_{n_0}^{N,1}(\la,n_0):=0.$
\end{enumerate}
\item After time $n_0$ we don't do anything, i.e. for any $n\in\{n_0+1,\ldots,N\}$
\begin{enumerate}
\item[]$\phi_n^{N,0}(\la,n_0):=\phi_{n_0}^{N,0}(\la,n_0),$
\item[]$\phi_{n}^{N,1}(\la,n_0):=\phi_{n_0}^{N,1}(\la,n_0)=0.$
\end{enumerate}

\end{itemize}
Note that the self-financing strategies given by $\{\phi^{N}(0,n_0)\}_{N\geq 1}$ correspond to the arbitrage strategies proposed by Sottinen in \cite{Sotti}. The next result extends this idea to the case of ``small'' transaction costs.

\begin{theorem}\label{Sot}
 If $\la_N=o\left(\frac1{\sqrt{N}}\right)$, then for all $N$ big enough, the $1$-step self-financing strategy $\phi^{N}(\lambda_N,N)$ leads to a $\la_N$-arbitrage in the $N$-fractional binary market.
\end{theorem}
\begin{proof}
As announced at the beginning of this section, if we define the self-financing system $\Phi^{N}(n_0):=\{\phi^{N}(\la,n_0)\}_{\la\in[0,1]}$, it is enough to show that 
 $$\lambda\left(\Phi^{N}(N)\right)\geq \frac{C}{\sqrt{N}},$$ 
 for some appropriate constant $C>0$. 

Note first that, for $n_0\geq n_H$, the value process of $\phi^{N}(\la,n_0)$ at maturity is given by 
$$V_N^{\la}\left(\phi^N(\la,n_0)\right)=(1-\la)S^{(N)}_{n_0-1}\1_{\{\vec{\xi}_{n_0-1}=-\vec{1}_{n_0-1}\}}-S^{(N)}_{n_0}\1_{\{\vec{\xi}_{n_0-1}=-\vec{1}_{n_0-1}\}}.$$

In order to have an arbitrage we need that 
$$V_N^{\la}\left(\phi^N(\la,n_0)\right)\geq 0\quad\textrm{a.s. and}\quad P\left(V_N^{\la}\left(\phi^N(\la,n_0)\right)>0\right)>0.$$ 
First, observe that
\begin{align*}
 V_N^{\la}\left(\phi^N(\la,n_0)\right)&=\1_{\{\vec{\xi}_{n_0-1}=-\vec{1}_{n_0-1}\}}\left((1-\la)S^{(N)}_{n_0-1}-S^{(N)}_{n_0}\right)\\
 & =S^{(N)}_{n_0-1}\1_{\{\vec{\xi}_{n_0-1}=-\vec{1}_{n_0-1}\}}\left(-\la-\frac{X_{n_0}(-\vec{1}_{n_0-1},\xi_{n_0})}{N^H}\right)\\
 & \geq S^{(N)}_{n_0-1}\1_{\{\vec{\xi}_{n_0-1}=-\vec{1}_{n_0-1}\}}\left(-\la-u_{n_0}^{(N)}(-\vec{1}_{n_0-1})\right),\\
\end{align*}
and then $V_N^{\la}\left(\phi^N(\la,n_0)\right)\geq0$ a.s. iff 
$$\la\leq-u_{n_0}^{(N)}(-\vec{1}_{n_0-1}).$$
Observe that the right-hand side is strictly positive by \eqref{albad}. Additionally, since $u_{n_0}^{(N)}(-\vec{1}_{n_0-1})>d_{n_0}^{(N)}(-\vec{1}_{n_0-1})$, for each $\la\leq -u_{n_0}^{(N)}(-\vec{1}_{n_0-1})$ we also have
\begin{equation}\label{prob}
P\left(V_N^{\la}\left(\phi^N(\la,n_0)\right)>0\right)\geq P\left(V_N^{\la}\left(\phi^N(\la,n_0)\right)\1_{\{\vec{\xi}_{n_0}=-\vec{1}_{n_0}\}}>0\right)=\frac1{2^{n_0}}>0.
\end{equation}
Therefore, we have
\begin{align}\label{laphin}
 \la\left(\Phi^{N}(n_0)\right)=-u_{n_0}^{(N)}(-\vec{1}_{n_0-1})=\frac1{N^H}\left(\sum_{i=1}^{n_0-1}j_{n_0}(i)-g_{n_0}\right).
\end{align}
Using Lemma \ref{ej1}, we see that
$$\sum_{i=1}^{n_0-1}j_{n_0}(i)\geq c_* (n_0-1)^{H-\frac{1}{2}}\int\limits_0^{n_0-1}x^{\frac{1}{2}-H}\phi_{n_0}(x)dx.$$
In addition, we have that
 $$\int\limits_{0}^{n_0-1} x^{\frac{1}{2}-H} (n_0-x)^{H-\frac{1}{2}} dx= n_0\int\limits_{0}^{\frac{n_0-1}{n_0}} u^{\frac{1}{2}-H} (1-u)^{H-\frac{1}{2}} du,$$
 and 
 $$\int\limits_{0}^{n_0-1} x^{\frac{1}{2}-H} (n_0-1-x)^{H-\frac{1}{2}} dx= (n_0-1)\int\limits_{0}^{1} u^{\frac{1}{2}-H} (1-u)^{H-\frac{1}{2}} du.$$
Thus, we obtain:
\begin{equation}\label{integ}
 \int\limits_{0}^{n_0-1} x^{\frac{1}{2}-H}\phi_{n_0}(x)dx=\int\limits_{0}^{1} u^{\frac{1}{2}-H} (1-u)^{H-\frac{1}{2}} du- n_0 \int\limits_{1-\frac{1}{n_0}}^{1} u^{\frac{1}{2}-H} (1-u)^{H-\frac{1}{2}} du.
 \end{equation}
Moreover, it is straightforward to show that there is a constant $\hat{c}>0$ such that
\begin{equation}\label{ineq}
0<n_0 \int\limits_{1-\frac{1}{n_0}}^{1} u^{\frac{1}{2}-H} (1-u)^{H-\frac{1}{2}} du \leq \frac{\hat{c}}{n_0^{H-\frac{1}{2}}}.
\end{equation}
Thus, for $n_0$ big enough and an appropriate constant $\tilde{c}>0$, we have that
$$\sum_{i=1}^{n_0-1}j_{n_0}(i)\geq \tilde{c}\, n_0^{H-\frac{1}{2}}.$$
Going back to \eqref{laphin} and using the estimates given in Lemma \ref{eg1}, we obtain:
\begin{align}\label{lbctstr}
  \la\left(\Phi^{N}(n_0)\right)\geq\hat{c}_* \frac{n_0^{H-\frac{1}{2}}}{N^H},
\end{align}
for some appropriate constant $\hat{c}_*>0$ and $n_0\geq n_H$ big enough. One can see from \eqref{lbctstr} that the bigger we choose $n_0$,  the better the lower bound becomes. In particular, when $n_0=N$, we have that
$$\la\left(\Phi^{N}(N)\right)\geq\frac{\hat{c}_*}{\sqrt{N}}.$$
The proof is now complete.
\end{proof}


\section{Asymptotic behavior of the critical transaction costs}\label{4}
The goal of this section is to study the asymptotic behavior of the critical transaction costs associated to the fractional binary markets.
More precisely, on each $N$-fractional binary market, we define $\la_c^{(N)}$ as in \eqref{lac}, i.e.
\begin{equation}\label{la}
\lambda_c^{(N)}=\inf\{\lambda\in[0,1]: \nexists\textrm{ $\la$-arbitrage in the $N$-fractional binary market}\},
\end{equation}
and we study the limit of $\lambda_c^{(N)}$ when $N$ tends to infinity.
Since the arbitrage opportunities disappear when we introduce arbitrarily small transaction costs in the fractional Black Scholes, meaning that the corresponding critical transaction cost is $0$, one can expect that also $\la_c^{(N)}\to0$ as $N\to\infty$. However, a completely opposite behavior will be observed.

Note that, from Remark \ref{ilplc} we have, for any self-financing system $\Phi^N$, that
\begin{equation}\label{ctc}
\la(\Phi^{N})\leq\la_c^{(N)}.
\end{equation}
Therefore, one way to achieve our goal, at least partially, would be to construct appropriate self-financing systems.
\begin{remark}
If we apply \eqref{ctc} to the self-financing system defined in the proof of Theorem \ref{Sot}, $\Phi^N(N)$, we deduce that
\begin{equation}\label{lbct}
\la_c^{(N)}\geq \frac{C}{\sqrt{N}},
\end{equation}
where $C>0$ is an appropriate constant. This result can be derived also using the lower bound \eqref{lowbd}. Indeed, the lower bound given in \eqref{lowbd} together with the estimations presented in Section~\ref{est} for the involved quantities lead us exactly to \eqref{lbct}. However, the advantage of the method given in Section \ref{3} is that explicit 1-step arbitrage opportunities under transaction costs are provided.
\end{remark}

Now, we proceed to the construction of a new sequence of self-financing systems $\{\Psi^N\}_{N\geq 1}$. Let's first go back to the frictionless case. Fix $\gamma\in(0,1)$ and assume that the price process takes jumps only down till time $\lfloor\gamma N\rfloor$. If $\lfloor\gamma N\rfloor\geq n_H$, we know from \eqref{albad} that the stock price process will decrease from time $\lfloor\gamma N\rfloor$ to time $\lfloor\gamma N\rfloor+1$. This is the idea behind the Sottinen's arbitrage strategy. However, we would like to do something better. We would like to prove that we can choose $k_N$ with $k_N\to\infty$ ensuring that the stock price will decrease until time $\lfloor\gamma N\rfloor+k_N$. If this is possible, we can construct an arbitrage by going short in one unit of stock at time $\lfloor\gamma N\rfloor$ and buying one unit at time $\lfloor\gamma N\rfloor+k_N$. Hence, we want to choose $k_N$ such that
\begin{align*}
u_{\lfloor\gamma N\rfloor+k}^{(N)}\left(-\vec{1}_{\lfloor\gamma N\rfloor},\vec{x}_{k-1}\right)=\frac{\Ys_{\lfloor\gamma N\rfloor+k}(-\vec{1}_{\lfloor\gamma N\rfloor},\vec{x}_{k-1})+g_{\lfloor\gamma N\rfloor+k}}{N^H}\leq\,0,
\end{align*}
for all $k\leq k_N$ and $\vec{x}_{k-1}\in\{-1,1\}^{k-1}$. Equivalently,
\begin{align*}
u_{\lfloor\gamma N\rfloor+k}^{(N)}(-\vec{1}_{\lfloor\gamma N\rfloor},\vec{1}_{k-1})=\frac{\Ys_{\lfloor\gamma N\rfloor+k}(-\vec{1}_{\lfloor\gamma N\rfloor},\vec{1}_{k-1})+g_{\lfloor\gamma N\rfloor+k}}{N^H}\leq\,0
\end{align*}
for any $k\leq k_N$. This is again equivalent to
\begin{align}\label{acres}
 A_{\gamma}^N(k):&=\Ys_{\lfloor\gamma N\rfloor+k}(-\vec{1}_{\lfloor\gamma N\rfloor},\vec{1}_{k-1})+g_{\lfloor\gamma N\rfloor+k}\nonumber\\
 &=-\sum_{i=1}^{\lfloor\gamma N\rfloor}j_{\lfloor\gamma N\rfloor+k}(i)+\sum_{i=\lfloor\gamma N\rfloor+1}^{\lfloor\gamma N\rfloor+k-1}j_{\lfloor\gamma N\rfloor+k}(i)+g_{\lfloor\gamma N\rfloor+k}\leq0,
\end{align}
for any $k\leq k_N$.

The next result tells us how $k_N$ has to be chosen in order for equation \eqref{acres} to hold true.

\begin{lemma}\label{arbi}
 For all $\gamma\in(0,1)$, there exist $P_{\gamma}\in(0,1-\gamma]$, $C_{\gamma},\widehat{C}_{\gamma}>0$ and $N_0^\gamma\in \Nb$ such that for all $N\geq N_0^{\gamma}$ and all $1\leq k\leq \lfloor P_{\gamma}N\rfloor$ the following holds:
 \begin{equation}\label{acres1}
  A_{\gamma}^{(N)}(k)\leq-C_{\gamma}N^{H-\frac{1}{2}}+\widehat{C}_{\gamma}\leq0.
 \end{equation}
\end{lemma}
\begin{proof}
First, we denote $\al:=H-\frac{1}{2}\in(0,\frac{1}{2})$, $\gamma_N:=\frac{\gn}{N}$ and we choose $n_\gamma>1$ such that, for all $N\geq n_\gamma$, we have $\gamma_N>\gamma/2$.

Now, we proceed to obtain a lower bound for the first term in \eqref{acres}. From Lemma \ref{ej1} we have, for all $1\leq k\leq N-\gn$, that
\begin{align*}
 \sum_{i=1}^{\gn}j_{\gn+k}(i)&\geq c_*\left(\gn+k-1\right)^{\al}\int\limits_0^{\gn}x^{-\al}\phi_{\gn+k}(x)dx.
 \end{align*}
After a change of variables, and denoting $I(x):=\int_0^x v^{-\al}(1-v)^{\al}dv$, we obtain that
 \begin{align}\label{paft}
 \int\limits_0^{\gn}&x^{-\al}\phi_{\gn+k}(x)dx=\nonumber\\
 &\qquad=\left(\gn+k\right)\,I\left(\frac{\gn}{\gn+k}\right)-\left(\gn+k-1\right)\,I\left(\frac{\gn}{\gn+k-1}\right)\nonumber\\
 &\qquad=N\left[\left(\gamma_N+\frac{k}{N}\right)\,I\left(\frac{\gamma_{N}}{\gamma_{N}+\frac{k}{N}}\right)-\left(\gamma_{N}+\frac{k}{N}-\frac1{N}\right)\,I\left(\frac{\gamma_{N}}{\gamma_{N}+\frac{k}{N}-\frac1{N}}\right)\right]\nonumber\\
 &\qquad=N\,F_{\frac{k}{N}}\left(\frac1{N}\right),
\end{align}
where, for $y\in(0,1)$, the function $F_y:(0,y)\rightarrow \Rb$ is given by:
\begin{equation*}
 F_y(h):=(\gamma_{N}+y)\,I\left(\frac{\gamma_{N}}{\gamma_{N}+y}\right)-(\gamma_{N}+y-h)\,I\left(\frac{\gamma_{N}}{\gamma_{N}+y-h}\right).
 \end{equation*}
 Note that
 \begin{equation}\label{fpeg}
 F'_y(h)=G\left(\frac{\gamma_{N}}{\gamma_{N}+y-h}\right), 
 \end{equation}
where $G:(0,1)\rightarrow \Rb$ is the function defined by $G(z):=I(z)-(1-z)^{\al}z^{1-\al}$. One can easily check that $G$ is strictly increasing in $(0,1)$. Consequently, we have, for all $N\geq n_\gamma$ and $y\leq 1-\gamma_N$, that
\begin{equation}\label{gmg}
G\left(\frac{\gamma_{N}}{\gamma_{N}+y-h}\right)\geq G(\gamma_{N})> G(\gamma/2)>G(0+)=0. 
\end{equation}
Plugging \eqref{fpeg} and \eqref{gmg} in \eqref{paft}, and using that $F_{\frac{k}{N}}(0)=0$, we obtain
\begin{align*}
 \int\limits_0^{\gn}x^{-\al}\phi_{\gn+k}(x)dx&=N\left[F_{\frac{k}{N}}\left(\frac1{N}\right)-F_{\frac{k}{N}}(0)\right]=N\int\limits_0^{\frac1{N}}F'_{\frac{k}{N}}(h)dh\\
 &> N\int\limits_0^{\frac1{N}}G(\gamma/2)dh=G(\gamma/2)>0.
\end{align*}
Returning to our initial inequality, we get, for all $N\geq n_\gamma$ and $k\leq N-\gn$, that
\begin{equation}\label{sum1}
 \sum_{i=1}^{\gn}j_{\gn+k}(i)\geq c_*(\gn+k-1)^{\al}G(\gamma/2)\geq c_*\,G(\gamma/2)\,\gn^{\al}.
\end{equation}
For the second sum in \eqref{acres}, we again use the estimates given in Lemma \ref{ej1} and after an appropriate change of variables we deduce that
\begin{align}\label{sum2}
 \sum_{i=\gn+1}^{\gn+k-1}j_{\gn+k}(i)&\leq c_*(\gn+k)^{\al}\int\limits_{\gn}^{\gn+k-1}x^{-\al}\phi_{\gn+k}(x)dx\nonumber\\
 &\leq c_*\frac{(\gn+k)^{\al}}{\gn^{\al}}\int\limits_{\gn}^{\gn+k-1}\phi_{\gn+k}(x)dx\nonumber\\
 &= c_*\frac{(\gn+k)^{\al}}{\gn^{\al}}\frac{[k^{\al+1}-1-(k-1)^{\al+1}]}{\al+1}\nonumber\\
 &\leq c_*\frac{(\gn+k)^{\al}}{\gn^{\al}}k^{\al}.
\end{align}
Using \eqref{sum1}, \eqref{sum2} and the upper bound for $g_{\gn+k}$ given in Lemma \ref{eg1}, we obtain, for all $N\geq n_\gamma$ and $1\leq k\leq N-\gn$, that 
\begin{align*}
 A_{\gamma}^N(k)&\leq -c_*\gn^{\al}G(\gamma/2)+c_*\frac{(\gn+k)^{\al}}{\gn^{\al}}k^{\al}+\frac{c_*}{\al+1}\frac{(\gn+k)^{\al}}{(\gn+k-1)^{\al}}\\
 &\leq c_*\left[-\gn^{\al}G(\gamma/2)+\gamma_N^{-\al}k^{\al}+\frac1{\al+1} \gamma_N^{-\al}\right]\\ 
 &\leq c_*\left[-N^{\al}\left((\gamma/2)^{\al}G(\gamma/2)-(\gamma/2)^{-\al}\left(\frac {k}{N}\right)^{\al}\right)+\frac1{\al+1}(\gamma/2)^{-\al}\right].
\end{align*}
Define $\widehat{C}_{\gamma}:=c_*(\gamma/2)^{-\al}/(\al+1)$. It is now clear that \eqref{acres1} is satisfied if $k\leq N-\gn$ and, for some $C_\gamma>0$, we have
$$(\gamma/2)^{\al}G(\gamma/2)-(\gamma/2)^{-\al}\left(\frac {k}{N}\right)^{\al}\geq \frac{C_\gamma}{c_*}.$$
If we choose $C_\gamma:=c_*(\gamma/2)^{\al}G(\gamma/2)/2$, the previous condition, can be written as
$$k\leq\,\left[\frac{(\gamma/2)^{2\al}G(\gamma/2)}{2}\right]^{\frac1{\al}}\,N.$$
Therefore, it is enough to define 
$$P_\gamma:=(1-\gamma)\wedge\left[\frac{(\gamma/2)^{2\al}G(\gamma/2)}{2}\right]^{\frac1{\al}}\quad\textrm{and}\quad N_0^\gamma:=n_\gamma\vee\left(\left\lfloor\left(\frac{\widehat{C}_\gamma}{C_\gamma}\right)^{\frac1{\alpha}}\right\rfloor+1\right),$$
to finish the proof.
\end{proof}

The above result tells us how long the stock price process will decrease starting from time $\lfloor\gamma N\rfloor$, given that until that time the price only jumped down. In what follows, using this knowledge, we construct a sequence of self-financing systems leading to good lower bounds for the critical transaction costs. 

Consider, for each $\la\in[0,1]$ and $N\geq 1$, the following $\la$-self-financing strategy $\psi^{N}(\la)=(\psi^{N,0}(\la),\psi^{N,1}(\la))$:
\begin{itemize}
 \item Before time $\lfloor\gamma N\rfloor$ we don't do anything, i.e. for any $i\in\{1,...,\lfloor\gamma N\rfloor-1\}$
 \begin{enumerate}
  \item[] $\psi_{i}^{N,0}(\la):=\psi_{i}^{N,1}(\la):=0$.
 \end{enumerate}
 \item  At time $\lfloor\gamma N\rfloor$ we short-sell one unit of stock, in which case 
 \begin{enumerate}
  \item[]$\psi_{\lfloor\gamma N\rfloor}^{N,0}(\la):=(1-\la)S^{(N)}_{\lfloor\gamma N\rfloor}\1_{\{\vec{\xi}_{\lfloor\gamma N\rfloor}=-\vec{1}_{\lfloor\gamma N\rfloor}\}},$
  \item[]$\psi_{\lfloor\gamma N\rfloor}^{N,1}(\la):=-\1_{\{\vec{\xi}_{\lfloor\gamma N\rfloor}=-\vec{1}_{\lfloor\gamma N\rfloor}\}}.$
 \end{enumerate}
\item Starting with $\lfloor\gamma N\rfloor+1$ we let the price evolve, and we don't do anything till $\lfloor\gamma N\rfloor+\lfloor P_\gamma N\rfloor$ when we liquidate the position, which means buying one unit of stock. In this case
\begin{enumerate}
  \item[]$\psi_{\lfloor\gamma N\rfloor+\lfloor P_\gamma N\rfloor}^{N,0}(\la):=\left((1-\la)S^{(N)}_{\lfloor\gamma N\rfloor}-S^{(N)}_{\lfloor\gamma N\rfloor+\lfloor P_\gamma N\rfloor}\right)\1_{\{\vec{\xi}_{\lfloor\gamma N\rfloor}=-\vec{1}_{\lfloor\gamma N\rfloor}\}},$
  \item[]$\psi_{\lfloor\gamma N\rfloor+\lfloor P_\gamma N\rfloor}^{N,1}(\la):=0.$
 \end{enumerate}
 \item Starting from time $\lfloor\gamma N\rfloor+\lfloor P_\gamma N\rfloor+1$ we don't do anything again, i.e. for all $i\in\{\lfloor\gamma N\rfloor+\lfloor P_\gamma N\rfloor+1,..,N\}$
 \begin{enumerate}
  \item[] $\psi_{i}^{N,0}(\la):=\psi_{\lfloor\gamma N\rfloor+\lfloor P_\gamma N\rfloor}^{N,0}(\la),$
  \item[] $\psi_{i}^{N,1}(\la):=\psi_{\lfloor\gamma N\rfloor+\lfloor P_\gamma N\rfloor}^{N,1}(\la)=0.$
 \end{enumerate}
\end{itemize}
Using the previous self-financing strategies, we define, on each $N$-fractional binary market, the self-financing system $\Psi^N:=\{\psi^N(\la)\}_{\la\in[0,1]}$. With the help of these self-financing systems, we obtain the following characterization of the asymptotic behavior of the critical transaction costs.
\begin{theorem}\label{lac1}
 There exists a constant $C_{\gamma}^*>0$ such that for $N$ big enough
 $$\la_c^{(N)}\geq 1-e^{-C_{\gamma}^* \sqrt{N}}.$$
 In particular, we have that 
 $$\lim_{N\rightarrow\infty}\lambda_c^{(N)}=1.$$ 
\end{theorem}
\begin{proof}
We start by looking, for each $\la\in[0,1]$ and $N\geq 1$, at the value process at maturity of the trading strategy $\psi^{N}(\la)$. Note that, if we set, for $j\geq\lfloor\gamma N\rfloor$, $\vec{\xi}_j^N=(\xi_{\lfloor\gamma N\rfloor+1},\ldots,\xi_{j})$ and $\vec{1}_j^N=\vec{1}_{j-\lfloor\gamma N\rfloor}$ (with the convention $\vec{\xi}_{\lfloor\gamma N\rfloor}^N=\vec{1}_{\lfloor\gamma N\rfloor}^N=\emptyset$),  we obtain
\begin{align*}
& V_N^{\la}(\psi^{N}(\la))=\1_{\{\vec{\xi}_{\lfloor\gamma N\rfloor}=-\vec{1}_{\lfloor\gamma N\rfloor}\}}\left[(1-\la)S^{(N)}_{\lfloor\gamma N\rfloor}-S^{(N)}_{\lfloor\gamma N\rfloor+\lfloor P_\gamma N\rfloor}\right]\\
 &=\1_{\{\vec{\xi}_{\lfloor\gamma N\rfloor}=-\vec{1}_{\lfloor\gamma N\rfloor}\}}S^{(N)}_{\lfloor\gamma N\rfloor}\left[1-\la-\prod_{j=\lfloor\gamma N\rfloor+1}^{\lfloor\gamma N\rfloor+\lfloor P_\gamma N\rfloor}\left(1+\frac{\Ys_j(-\vec{1}_{\gn},\vec{\xi}_{j-1}^N)+g_j\xi_j}{N^H}\right)\right]\\
&\geq\1_{\{\vec{\xi}_{\lfloor\gamma N\rfloor}=-\vec{1}_{\lfloor\gamma N\rfloor}\}}S^{(N)}_{\lfloor\gamma N\rfloor}\left[1-\la-\prod_{j=\lfloor\gamma N\rfloor+1}^{\lfloor\gamma N\rfloor+\lfloor P_\gamma N\rfloor}\left(1+\frac{\Ys_j(-\vec{1}_{\gn},\vec{1}_{j-1}^N)+g_j}{N^H}\right)\right]\\
 &=\1_{\{\vec{\xi}_{\lfloor\gamma N\rfloor}=-\vec{1}_{\lfloor\gamma N\rfloor}\}}S^{(N)}_{\lfloor\gamma N\rfloor}\left[1-\la-\prod_{k=1}^{\lfloor P_\gamma N\rfloor}\left(1+\frac{A_{\gamma}^{N}(k)}{N^H}\right)\right].
\end{align*}
The previous inequality is an equality if and only if 
$$\vec{\xi}_{\lfloor\gamma N\rfloor}\neq-\vec{1}_{\lfloor\gamma N\rfloor}\quad\textrm{or}\quad\vec{\xi}_{\lfloor\gamma N\rfloor+\lfloor P_\gamma N\rfloor}=\left(-\vec{1}_{\gn},\vec{1}_{\lfloor\gamma N\rfloor+\lfloor P_\gamma N\rfloor}^N\right).$$ 
Consequently, $V_N^{\la}(\psi^{N}(\la))\geq0$ $a.s.$ if and only if
$$1-\la-\prod_{k=1}^{\lfloor P_\gamma N\rfloor}\left(1+\frac{A_{\gamma}^{N}(k)}{N^H}\right)\geq0.$$
Or equivalently,
$$\la\leq1-\prod_{k=1}^{\lfloor P_\gamma N\rfloor}\left(1+\frac{A_{\gamma}^{N}(k)}{N^H}\right).$$
In which case, we have
\begin{equation*}
P\left(V_N^{\la}(\psi^{N}(\la))>0\right)\geq \frac1{2^{\lfloor\gamma N\rfloor}}-\frac1{2^{\lfloor\gamma N\rfloor+\lfloor P_\gamma N\rfloor}}>0,
\end{equation*}
and therefore, $\psi^{N}(\la)$ provides a $\la$-arbitrage. Consequently, we obtain that
\begin{align*}
\la(\Psi^{N})=1-\prod_{k=1}^{\lfloor P_\gamma N\rfloor}\left(1+\frac{A_{\gamma}^{N}(k)}{N^H}\right),
\end{align*}
and by \eqref{ctc} it follows that
$$\la_c^{(N)}\geq1-\prod_{k=1}^{\lfloor P_\gamma N\rfloor}\left(1+\frac{A_{\gamma}^{N}(k)}{N^H}\right).$$
Using Lemma \ref{arbi}, we deduce that
\begin{align*}
 \prod_{k=1}^{\lfloor P_\gamma N\rfloor}\left(1+\frac{A_{\gamma}^{N}(k)}{N^H}\right)&\leq \prod_{k=1}^{\lfloor P_\gamma N\rfloor}\left(1-\frac{C_{\gamma}\,N^{H-\frac{1}{2}}}{N^H}+\frac{\widehat{C}_{\gamma}}{N^H}\right)\\
 &=\left(1-\frac{C_{\gamma}}{\sqrt{N}}+\frac{\hat{C}_{\gamma}}{N^H}\right)^{\lfloor P_{\gamma} N\rfloor}\\
 &=e^{\lfloor P_{\gamma} N\rfloor\ln\left(1-\frac{C_{\gamma}}{\sqrt{N}}+\frac{\widehat{C}_{\gamma}}{N^H}\right)}\\
 &\leq e^{-C_{\gamma}^* \sqrt{N}},
\end{align*}
for some well-chosen constant $C_{\gamma}^*>0$ and $N$ sufficiently large. This concludes the proof.
\end{proof}

\subsection*{Conclusions}
In this section, we have proved that the sequence of critical transaction costs in the fractional binary markets converges to $1$. More precisely, we have constructed an explicit sequence of self-financing systems $\{\Psi^N\}_{N\geq 1}$ verifying that
$$\lambda_c^{(N)}\geq \lambda(\Psi^N)\xrightarrow[N\rightarrow\infty]{} 1.$$
In particular, for each $\lambda\in(0,1)$, the trading strategy $\psi^N(\la)$ provides a $\la$-arbitrage in the $N$-fractional binary market, when $N$ is sufficiently large. As pointed out in the introduction, this result is in apparent contradiction with the fact that the fractional binary markets approximate the fractional Black-Scholes model, which is free of arbitrage under arbitrarily small transaction costs. To explain this, we first note that
\begin{equation*}
P\left(V_N^{\la}(\psi^{N}(\la))>0\right)\leq\frac1{2^{\lfloor\gamma N\rfloor}}\xrightarrow[N\rightarrow\infty]{}0.
\end{equation*}
In other words, the probability of getting a strictly positive profit using the trading strategy $\psi^N(\la)$ vanishes in the limit when $N$ tends to $\infty$. Additionally, we observe that
$$0<\1_{\{\vec{\xi}_{\lfloor\gamma N\rfloor}=-\vec{1}_{\lfloor\gamma N\rfloor}\}}S^{(N)}_{\lfloor\gamma N\rfloor}\leq\left(1-\frac{g}{N^H}\right)^{\lfloor\gamma N\rfloor}\xrightarrow[N\rightarrow\infty]{} 0,$$
which implies that
$$0\leq V_N^{\la}(\psi^{N}(\la))\leq \1_{\{\vec{\xi}_{\lfloor\gamma N\rfloor}=-\vec{1}_{\lfloor\gamma N\rfloor}\}}S^{(N)}_{\lfloor\gamma N\rfloor}\xrightarrow[N\rightarrow\infty]{} 0.$$
This means that the profit obtained using the trading strategy $\psi^{N}(\la)$ converges to zero. 

Summarizing, the self-financing strategies constructed in this section provide arbitrage opportunities in the fractional binary markets under, arbitrarily close to one, transaction costs. However, the probability of getting such an arbitrage and the magnitude of the corresponding gain, both converge to zero. Stated differently, the arbitrage opportunities disappear in the limit. Therefore, our results are not in contradiction with the behavior under friction of the fractional Black-Scholes model. This also tells us that, even though the notion of critical transaction cost permits to characterize the existence of arbitrage in a fixed binary market, the study of critical transaction costs in a sequence does not imply any information about the behavior under friction of an eventual limit market. Moreover, any suitable notion of ``sequential'' arbitrage should impose that: (1) the probabilities of getting  strictly positive profits are uniformly bounded from below by a strictly positive constant, and (2) the corresponding gains are uniformly strictly positive. The concept of asymptotic arbitrage introduced in Section \ref{2} satisfies these conditions and then, it is worth to study the fractional binary markets from this perspective.  
\section{1-step asymptotic arbitrage opportunities}\label{5}
In this section, we aim to study the existence of asymptotic arbitrage opportunities in the fractional binary markets, when we are constrained to use only sequences of 1-step self-financing strategies (as defined in Remark \ref{1step}). When such an opportunity exists, we call it 1-step asymptotic arbitrage.

From its definition, a 1-step $\la$-self-financing strategy in the $N$-fractional binary market can be expressed, for some fixed
$n\in\{1,...,N-1\}$, $A\subset \{-1,1\}^n$ and $q:\{-1,1\}^n\rightarrow\Rb_+$, as follows:
\begin{align*}
\phi_i^{0,N}&=\phi_i^{1,N}=0,\quad \textrm{for all } i<n,\\
\phi_n^{0,N}&=\sum\limits_{\vec{x}\in A}1_{\{\vec{\xi}_n=\vec{x}\}}(1-\la)S_n^{(N)}(\vec{x})q(\vec{x})-\sum\limits_{\vec{x}\in A^c}1_{\{\vec{\xi}_n=\vec{x}\}}S_n^{(N)}(\vec{x})q(\vec{x}),\\
\phi_n^{1,N}&=-\sum\limits_{\vec{x}\in A}1_{\{\vec{\xi}_n=\vec{x}\}}q(\vec{x})+\sum\limits_{\vec{x}\in A^c}1_{\{\vec{\xi}_n=\vec{x}\}}q(\vec{x}),\\
\phi_{n+1}^{0,N}&=\sum\limits_{\vec{x}\in A}1_{\{\vec{\xi}_n=\vec{x}\}}\left[(1-\la)S_n^{(N)}(\vec{x})-S_{n+1}^{(N)}(\vec{x},\xi_{n+1})\right]q(\vec{x})\\
&\quad-\sum\limits_{\vec{x}\in A^c}1_{\{\vec{\xi}_n=\vec{x}\}}\left[S_n^{(N)}(\vec{x})-(1-\la)S_{n+1}^{(N)}(\vec{x},\xi_{n+1})\right]q(\vec{x})\\
\phi_{n+1}^{1,N}&=0,\\
\phi_k^{0,N}&=\phi_{n+1}^{0,N}\quad \textrm{for all } n+1<k\leq N,\\
\phi_k^{1,N}&=\phi_{n+1}^{1,N},\quad \textrm{for all } n+1<k\leq N.
\end{align*}

In other words, we don't do anything before time $n$; at time $n$, depending on the position $\vec{x}$, we go short or long in $q(\vec{x})$ units of stock; at time $n+1$ we liquidate the position; and after $n+1$ we don't do anything.

In particular, the value process at maturity, is given by 
\begin{align}\label{gvp}
 V_N^\la(\phi^N)&=\sum\limits_{\vec{x}\in A}1_{\{\vec{\xi}_n=\vec{x}\}}\left[(1-\la)S_n^{(N)}(\vec{x})-S_{n+1}^{(N)}(\vec{x},\xi_{n+1})\right]q(\vec{x})\nonumber\\
&\quad-\sum\limits_{\vec{x}\in A^c}1_{\{\vec{\xi}_n=\vec{x}\}}\left[S_n^{(N)}(\vec{x})-(1-\la)S_{n+1}^{(N)}(\vec{x},\xi_{n+1})\right]q(\vec{x}).
\end{align}

\subsection{1-step asymptotic arbitrage of first kind under small transaction costs}
In this section, we construct a 1-step asymptotic arbitrage when the transaction costs converge fast enough to zero. Our construction is based on the 1-step self-financing systems given in Section \ref{3}. Note first that, as shown in Section \ref{3}, the trading strategies $\{\phi^N(\la_N,N)\}_{N\geq 1}$ verify, when $\la_N=o(1/\sqrt{N})$, that
\begin{equation*}
P\left(V_N^{\la_N}(\phi^{N}(\la_N,N))>0\right)=\frac1{2^{N-1}}\xrightarrow[N\rightarrow\infty]{}0,
\end{equation*}
and then, they can not provide a 1-step asymptotic arbitrage. A different behavior is obtained, when we use the self-financing strategies $\{\phi^N(\la_N,n_H)\}_{N\geq 1}$ which verify, this time under transaction costs $\la_N=o(1/N^H)$, that
\begin{equation*}
P\left(V_N^{\la_N}(\phi^{N}(\la_N,n_H))>0\right)=\frac1{2^{n_H-1}}>0.
\end{equation*}
However, from the calculations in the proof of Theorem \ref{Sot}, it is straightforward to see that
$$V_N^{\la_N}(\phi^{N}(\la_N,n_H))\leq \frac{C}{N^H},$$
which means that the profit vanishes in the limit. On the other hand, in an asymptotic arbitrage of first kind, we have to get arbitrarily rich with a strictly positive probability. Thus, the trading strategies $\{\phi^N(\la_N,n_H)\}_{N\geq 1}$ do not provide such an asymptotic arbitrage. This problem can be solved, if instead of going short in 1-unit of stock, we go short in a sufficiently large amount of units of stock. However, we have to be careful, because if this amount of units of stock is too big, the admissibility condition can fail. These ideas are reflected in the next result.
\begin{theorem} Consider $\la_N=o(1/N^H)$ and let ${\{q_N\}}_{N\geq 1}$ be  a sequence of strictly positive numbers, verifying that
$$\frac{q_N}{N^H}\xrightarrow[N\rightarrow\infty]{}\infty\quad\textrm{and}\quad \la_N q_N\xrightarrow[N\rightarrow\infty]{}0.$$
The self-financing strategies $\{\widehat{\phi}^N\}_{N\geq 1}$, defined by $\widehat{\phi}^N:=q_N \,\phi^{N}(\la_N,n_H)$, provide a 1-step ${\{\la_N\}}_{N\geq 1}$-asymptotic arbitrage of first kind.
\end{theorem}
\begin{proof}
Let's fix ${\{\la_N\}}_{N\geq 1}$ and ${\{q_N\}}_{N\geq 1}$ as in the statement. Note that the value process of the self-financing strategy $\widehat{\phi}^N$
satisfies
$$V_i^{\la_N}(\widehat{\phi}^{N})=q_N\,V_i^{\la_N}(\phi^{N}(\la_N,n_H)),\quad{i=1,...,N}.$$
In particular, from the estimations in the proof of Theorem \ref{Sot}, we obtain that
\begin{equation}\label{vpmat}
 V_N^{\la_N}(\widehat{\phi}^{N})\geq q_N\, S^{(N)}_{n_H-1}\1_{\{\vec{\xi}_{n_H-1}=-\vec{1}_{n_H-1}\}}\left(-\la_N+\frac{\theta_{n_H}}{N^H}\right),
\end{equation}
where $\theta_{n_H}=\sum_{i=1}^{n_H-1}j_{n_H}(i)-g_{n_H}$, which, from \eqref{albad}, is strictly positive.
From the properties of $\la_N$ and $q_N$ it follows that
$$q_N\left(-\la_N+\frac{\theta_{n_H}}{N^H}\right)\xrightarrow[N\rightarrow\infty]{}\infty.$$
Additionally, we have that
$$S^{(N)}_{n_H-1}(-\vec{1}_{n_H-1})=\prod\limits_{k=1}^{n_H-1}\left(1-\frac{\left(\sum\limits_{i=1}^{k-1}j_{k}(i)+g_{k}\right)}{N^H}  \right)s_0\xrightarrow[N\rightarrow\infty]{}s_0.$$
Therefore, if we define
$$C_N:=q_N\left(-\la_N+\frac{\theta_{n_H}}{N^H}\right)\,S^{(N)}_{n_H-1}(-\vec{1}_{n_H-1})\xrightarrow[N\rightarrow\infty]{}\infty,$$
we deduce that
$$P\left(V_N^{\la_N}(\widehat{\phi}^{N})\geq C_N\right)=\frac{1}{2^{n_H-1}}.$$
It remains only to check the admissibility conditions. Before time $n_H-1$, the value process is zero and then any admissibility condition is verified. Similarly, after time $n_H-1$ the value process is equal to the value process at maturity, which from \eqref{vpmat} is bigger or equal than zero, at least for $N$ large enough. Thus, we have only to check an appropriate admissibility at time $n_H-1$. Note that
$$ V_{n_H-1}^{\la_N}(\widehat{\phi}^{N})=-\la_N\,q_N\,S^{(N)}_{n_H-1}(-\vec{1}_{n_H-1})\1_{\{\vec{\xi}_{n_H-1}=-\vec{1}_{n_H-1}\}},$$
and then, if we define $c_N:=\la_N\,q_N\,S^{(N)}_{n_H-1}(-\vec{1}_{n_H-1})$, we have that $c_N\xrightarrow[N\rightarrow\infty]{} 0$ and that, the self-financing strategy $\widehat{\phi}^{N}$ is $c_N$-admissible. The proof is then concluded.
\end{proof}
\subsection{Absence of asymptotic arbitrage when the transaction costs converge slowly to zero}
Now, we are interested to provide a condition, in the sequence of transaction costs, avoiding the 1-step asymptotic arbitrage opportunities. The next result gives us an important estimate in order to deal with this problem.
\begin{lemma}\label{ubx}
There is a constant $c_X>0$, such that, for all $n>1$, we have
$$|X_n|\leq c_X\, n^{H-\frac{1}{2}}.$$
\end{lemma}
\begin{proof}
 The result is obtained following a similar calculation as in Theorem \ref{Sot}. More precisely, for all $n\geq2$, using \eqref{integ}, \eqref{ineq} and the upper estimate given in Lemma~\ref{ej1}, we have that
 \begin{align*}
 |X_n|&\leq \sum_{i=1}^{n-1}j_n(i)+g_n\leq c_*n^{H-\frac12}\int_0^{n-1}x^{\frac12-H}\phi_n(x)dx+g_n\\
 &=c_*n^{H-\frac12}\left(\int\limits_{0}^{1} u^{\frac{1}{2}-H} (1-u)^{H-\frac{1}{2}} du- n \int\limits_{1-\frac{1}{n}}^{1} u^{\frac{1}{2}-H} (1-u)^{H-\frac{1}{2}}du\right)+g_n\\
 &< c_*n^{H-\frac12}\int\limits_{0}^{1} u^{\frac{1}{2}-H} (1-u)^{H-\frac{1}{2}} du+g2^{H-\frac12}\\
 &\leq \frac{c_*}{\frac32-H}n^{H-\frac12}+gn^{H-\frac12}=c_Xn^{H-\frac12},
 \end{align*}
where $c_X:=\frac{c_*}{\frac32-H}+g$.
\end{proof}
\begin{theorem}
If the sequence of positive real numbers ${\{\la_N\}}_{N\geq 1}$ verifies that 
$$\la_N \sqrt{N}\xrightarrow[N\rightarrow\infty]{}\infty,$$ 
then there are neither 1-step ${\{\la_N\}}_{N\geq 1}$-asymptotic arbitrage opportunities of first kind nor 1-step ${\{\la_N\}}_{N\geq 1}$-asymptotic arbitrage opportunities of second kind.  
\end{theorem}
\begin{proof}
 Consider ${\{\la_N\}}_{N\geq 1}$ as in the statement and,  for each $N\geq 1$, a 1-step $\la_N$-self-financing strategy $\phi^N$. From the discussion at the beginning of this section and \eqref{gvp}, we know that the value process of $\phi^N$ at maturity has the following form
 \begin{align*}
  V_N^{\la_N}(\phi^N)&=\sum\limits_{\vec{x}\in A_N}1_{\{\vec{\xi}_{n_N}=\vec{x}\}}\left[(1-\la_N)S_{n_N}^{(N)}(\vec{x})-S_{n_N+1}^{(N)}(\vec{x},\xi_{n_N+1})\right]q_N(\vec{x})\\
&\quad-\sum\limits_{\vec{x}\in A_N^c}1_{\{\vec{\xi}_{n_N}=\vec{x}\}}\left[S_{n_N}^{(N)}(\vec{x})-(1-\la_N)S_{n_N+1}^{(N)}(\vec{x},\xi_{n_N+1})\right]q_N(\vec{x}).
\end{align*}
 where $n_N\in\{1,...,N-1\}$, $A_N\subset \{-1,1\}^{n_N}$ and $q_N:\{-1,1\}^{n_N}\rightarrow\Rb_+$. In particular, we have, for all $\vec{x}\in A_N$, that
 $$V_N^{\la_N}(\phi^N)1_{\{\vec{\xi}_{n_N}=\vec{x}\}}=1_{\{\vec{\xi}_{n_N}=\vec{x}\}}S_{n_N}^{(N)}(\vec{x})\left[-\la_N-\frac{X_{n_N+1}(\vec{x},\xi_{n_N+1})}{N^H}\right]q_N(\vec{x}).$$
Similarly, for each $\vec{x}\in A_N^c$, we have
  $$V_N^{\la_N}(\phi^N)1_{\{\vec{\xi}_{n_N}=\vec{x}\}}=1_{\{\vec{\xi}_{n_N}=\vec{x}\}}S_{n_N}^{(N)}(\vec{x})\left[-\la_N+\frac{(1-\la_N)X_{n_N+1}(\vec{x},\xi_{n_N+1})}{N^H}\right]q_N(\vec{x}).$$
Using the two previous identities and Lemma \ref{ubx}, we obtain that
 $$V_N^{\la_N}(\phi^N)\leq S_{n_N}^{(N)}(\vec{\xi}_{n_N})\left(-\la_N+\frac{c_X}{\sqrt{N}}\right)q_N(\vec{\xi}_{n_N}).$$
From the asymptotic behavior of $\la_N$, we conclude, for all $N$ sufficiently large, that
$$V_N^{\la_N}(\phi^N)\leq 0\quad \textrm{a.s.}.$$
And this implies that the sequence $\{\phi^N\}_{N\geq 1}$ does not provide an asymptotic arbitrage of first kind nor an asymptotic arbitrage of second kind. The result follows from the arbitrariness of this sequence.
\end{proof}
\begin{remark}
 The proof of the previous proposition also shows that, if 
 $$\la_N \sqrt{N}\xrightarrow[N\rightarrow\infty]{}\infty,$$
 then for all $N$ sufficiently large, the $N$-fractional binary market is free of 1-step arbitrage opportunities. In particular, if we define
 $$\la_{c,1}^{(N)}:=\inf \{\la\in[0,1]: \nexists\textrm{ 1-step $\la$-arbitrage in the $N$-fractional binary market }\},$$
 then we have that
 $$\la_{c,1}^{(N)}\xrightarrow[N\rightarrow\infty]{}0.$$
\end{remark}

\subsection{Absence of asymptotic arbitrage of second kind in the frictionless case}
In this section, we study the asymptotic arbitrage opportunities of second kind in the frictionless case. More precisely, we will prove that such opportunities do not exist when the Hurst parameter $H$ is close enough to $1/2$. 

First, let us introduce the following sets
$$\As_n^{u,H}:=\{\vec{x}\in {\{-1,1\}}^{n-1}: u_n^{(N)}(\vec{x})\leq 0\}=\{\vec{x}\in {\{-1,1\}}^{n-1}: \sum\limits_{i=1}^{n-1}j_{n}(i)x_i+g_{n}\leq 0\},$$
and
$$\As_n^{d,H}:=\{\vec{x}\in {\{-1,1\}}^{n-1}: d_n^{(N)}(\vec{x})\geq 0\}=\{\vec{x}\in {\{-1,1\}}^{n-1}: \sum\limits_{i=1}^{n-1}j_{n}(i)x_i-g_{n}\geq 0\}.$$
The set $\As_n^H:=\As_n^{u,H}\cup\As_n^{d,H}$ is the set of arbitrage points at level $n$ in the fractional binary markets with Hurst parameter $H$. Note that these sets do not depend on $N$, and the reason is that the sign of $u_n^{(N)}$ and $d_n^{(N)}$ does not depend on $N$. This simplification comes from the fact that we are treating the fractional binary markets without drift, i.e. $a_n^{(N)}=0$.

Now define
$$\nu_H:=\sup\limits_{n\geq 1}\frac{|\As_n^H|}{2^{n-1}}\leq 1\quad\textrm{and}\quad H_*:=\inf\{H\in(1/2,1]: \nu_H=1\}.$$
\begin{lemma}\label{apap}
We have that  $H_*\in(1/2,1]$.
\end{lemma}
\begin{proof}
It is straightforward to see that
$$|\As_1^H|=0.$$
On the other hand, we know from \cite{CKP} that for $n>1$
$$Var(\Ys_n)\leq \sigma^2\left(1- \frac{c_H^2}{(H+\frac{1}{2})^2}\right),$$
and that 
$$\frac{|\As_n^H|}{2^{n-1}}=P(|\Ys_n|\geq g_n).$$
Therefore, using Tchebysheff's inequality and Lemma \ref{eg1}, we deduce that
$$\frac{|\As_n^H|}{2^{n-1}}\leq \frac{\left(H+\frac12\right)^2}{c_H^2}-1.$$
From \eqref{ch} and the properties of the Gamma function, it follows that $c_H$ converges to $1$ when $H$ tends to $1/2$. 
As a consequence, we obtain that
$$\lim\limits_{H\rightarrow\frac{1}{2}}\,\sup\limits_{n\geq 1}\frac{|\As_n^H|}{2^{n-1}}=0.$$
The result follows from the continuity of the function $H\in(1/2,1)\mapsto c_H$.
\end{proof}

\begin{theorem}\label{aa2fc}
If $H\in(1/2,H_*)$, there is no 1-step asymptotic arbitrage of second kind in the sequence of fractional binary markets without friction. 
\end{theorem}
\begin{proof}
Fix $H\in(1/2,H_*)$ and let $\{\phi^N\}_{N\geq 1}$ be a sequence of 1-step self-financing strategies. We know from \eqref{gvp} that the value process of $\phi^N$ at maturity has the following form
 \begin{align*}
  V_N(\phi^N)&=\sum\limits_{\vec{x}\in A_N}1_{\{\vec{\xi}_{n_N}=\vec{x}\}}\left[S_{n_N}^{(N)}(\vec{x})-S_{n_N+1}^{(N)}(\vec{x},\xi_{n_N+1})\right]q_N(\vec{x})\\
&\quad-\sum\limits_{\vec{x}\in A_N^c}1_{\{\vec{\xi}_{n_N}=\vec{x}\}}\left[S_{n_N}^{(N)}(\vec{x})-S_{n_N+1}^{(N)}(\vec{x},\xi_{n_N+1})\right]q_N(\vec{x}).
\end{align*}
where $n_N\in\{1,...,N-1\}$, $A_N\subset \{-1,1\}^{n_N}$ and $q_N:\{-1,1\}^{n_N}\rightarrow\Rb_+$. In particular, we have, for all $\vec{x}\in A_N$, that
 $$V_N(\phi^N)1_{\{\vec{\xi}_{n_N}=\vec{x}\}}=-1_{\{\vec{\xi}_{n_N}=\vec{x}\}}\frac{\left(\sum\limits_{i=1}^{n_N}j_{n_N+1}(i)x_i+g_{n_N+1}\xi_{n_N+1}\right)}{N^H}S_{n_N}^{(N)}(\vec{x})q_N(\vec{x}).$$
Similarly, for each $\vec{x}\in A_N^c$, we have
  $$V_N(\phi^N)1_{\{\vec{\xi}_{n_N}=\vec{x}\}}=1_{\{\vec{\xi}_{n_N}=\vec{x}\}}\frac{\left(\sum\limits_{i=1}^{n_N}j_{n_N+1}(i)x_i+g_{n_N+1}\xi_{n_N+1}\right)}{N^H}S_{n_N}^{(N)}(\vec{x})q_N(\vec{x}).$$
Now we analyze the different situations. If $\vec{x}\in A_N\cap\As_{n_N+1}^{u,H}$ or $\vec{x}\in A_N^c\cap\As_{n_N+1}^{d,H}$, we have
$$P\left(\{V_N(\phi^N)\geq 0\}\cap{\{\vec{\xi}_{n_N}=\vec{x}\}}\right)=\frac{1}{2^{n_N}}.$$
If $\vec{x}\in A_N\cap\As_{n_N+1}^{d,H}$ or $\vec{x}\in A_N^c\cap\As_{n_N+1}^{u,H}$, 
$$P\left(\{V_N(\phi^N)> 0\}\cap{\{\vec{\xi}_{n_N}=\vec{x}\}}\right)=0.$$
If $\vec{x}\in \left(\As_{n_N+1}^H\right)^c$, then
$$P\left(\{V_N(\phi^N)\geq 0\}\cap{\{\vec{\xi}_{n_N}=\vec{x}\}}\right)=\frac{1}{2^{n_N+1}}.$$
Consequently, for any $\alpha>0$, we have
\begin{align*}
 P\left(V_N(\phi^N)\geq \alpha\right)&\leq \frac{|\As_{n_N+1}^H|}{2^{n_N}}+ \frac{|\left(\As_{n_N+1}^H\right)^c|}{2^{n_N+1}}\\
 &=\frac{1}{2}+\frac{1}{2}\times \frac{|\As_{n_N+1}^H|}{2^{n_N}}.
\end{align*}
Thus, using Lemma \ref{apap}, we deduce that
$$P\left(V_N(\phi^N)\geq \alpha\right)<\frac{1+\nu_H}{2}<1.$$
Therefore, we conclude that $\{\phi^N\}_{N\geq 1}$ does not provide an asymptotic arbitrage of second kind. The proposition follows from the generality of the 1-step self-financing strategies $\{\phi^N\}_{N\geq 1}$.
\end{proof}
\begin{remark}
The problem of determining if $H_*<1$ is, at our knowledge, still an open problem.  
\end{remark}

\subsection*{Future research}
 Having studied the existence of 1-step asymptotic arbitrage opportunities, it seems natural to analyze the asymptotic arbitrage under more general sequences of trading strategies. In particular, one could be interested in the existence of a stronger form of asymptotic arbitrage. This is the content of a forthcoming paper.

\bibliographystyle{plain}
\bibliography{reference}
\end{document}